\documentclass{amsart}

\usepackage{graphicx}        

\newtheorem{thm}{Theorem}[section]
\newtheorem{lem}[thm]{Lemma}

\numberwithin{equation}{section}

\begin{document}

\title{Acyclic Subgraphs in $k$-Majority Tournaments}

\author{Alexandra Ilic}

\address{Westlake High School, Austin, TX 78746}

\author{Lilly Shen}

\address{Clements High School, Sugar Land, TX 77479}

\author{Bobby Shen}

\address{Dulles High School, Sugar Land, TX 77478}

\author{Jian Shen}

\address{Department of
Mathematics, Texas State
University, San Marcos, TX 78666}

\email{js48@txstate.edu}

\thanks{Corresponding author. J.~Shen was partially supported by NSF (CNS 0835834, DMS 1005206) and Texas Higher Education Coordinating Board (ARP 003615-0039-2007).}


\maketitle

\begin{abstract}
A $k$-majority digraph is a directed graph created by combining $k$ individual rankings on the same ground set to form a consensus where edges point in the direction indicated by a strict majority of the rankings.
 The $k$-majority digraph is used to model voting scenarios, where the vertices correspond to options ranked by $k$ voters. When $k$ is odd, the resulting digraph is always a tournament, called $k$-majority tournament. Let $f_k(n)$ be the minimum,
over all $k$-majority tournaments with $n$ vertices, of the maximum order of an induced transitive sub-tournament. Recently, Milans, Schreiber, and West proved that $\sqrt n \le f_3(n) \le 2 \sqrt n +1 $. In this paper, we improve the upper bound of $f_3(n)$ by showing that $f_3(n) < \sqrt {2n} +\frac 12 $.
\end{abstract}

\section{Introduction}
Let $\Pi =\{ \pi _1, \ldots, \pi_k\}$ be a set of $k$ linear orders on a ground set $V$.
The {\em $k$-majority digraph} of $\Pi$ has vertex set $V$ and, for any two vertices $u$ and $v$ in $V$, there is an edge from $u$ to $v$ if and only if a strict majority of these
$k$ linear orders rank $u$ before $v$.
The $k$-majority digraph is used to model voting scenarios, where the vertices correspond to options ranked by $k$ voters.
When $k$ is odd, the resulting digraph is always a tournament, called {\em $k$-majority tournament}.

\vspace{3mm}

A set of vertices in a directed graph $G$ is call {\em acyclic} if the subgraph induced
by the set contains no cycle. Let
$a(G)$ denote the maximum size of an acyclic set in $G$. Recently,  Milans, Schreiber, and West \cite{MilansSW} defined the following parameter:
$$f_k(n) = \min \{ a(T): T \mbox{ is an } n \mbox{-vertex } k \mbox{-majority tournament} \};$$
that is, $f_3(n)$ is the minimum,
over all $k$-majority tournaments with $n$ vertices, of the maximum order of an induced transitive sub-tournament. Milans, Schreiber, and West \cite{MilansSW} proved that
$\sqrt n \le f_3(n) \le 2 \sqrt n +1 $. In order to prove their lower bound on $f_3(n)$, Milans, Schreiber, and West made the following definition. A set $X \subseteq V$ is called $\Pi$-{\em consistent} if it appears in the same order in each linear order of $\Pi$; and similarly, $X$ is called $\Pi$-{\em neutral} if $|\Pi|=k$ is even and, for all distinct $u$, $v$ in $X$, the element $u$ appears before $v$ in exactly half the members of $\Pi$. The lower bound on $f_3(n)$ then was proved by
applying the following Erd\"os-Szekeres theorem.

\begin{thm} [Erd\"os-Szekeres \cite{ErdosS}]  \label{condition}
Given linear orders $\pi_1$ and $\pi_2$ of a set $V$ with $|V| > (r-1)(s-1),$
there is either a $\{\pi_1 , \pi_2\}$-consist set of size $r$ or a $\{\pi_1 , \pi_2\}$-neutral set of size $s$.
\end{thm}

To prove their upper bound on $f_3(n)$, Milans, Schreiber, and West \cite{MilansSW}
considered a special case when $n =r^2$ is a perfect square and
constructed three linear orders on $r^2$ lattice points arranged by a square with a side length of $(r-1)$. In this paper, we improve the upper bound for $f_3(n)$ by showing that $f_3(n) < \sqrt {2n} +\frac 12 $. Our proof strategy is to consider a special case when $n =r(r+1)/2 $ and
construct three linear orders on $r(r+1)/2$ hexagonal lattice points arranged by an equivalent triangle with a side length of $(r-1)$.

\section{Proof of Main Result}
The hexagonal lattice is a regular, repeating grid of points in the plane in which each point in the lattice is $1$ unit away from each of its $6$ nearest neighbors.
Let $T_r$ be an equilateral triangular portion of the hexagonal lattice with a side length of $(r-1)$ units. So $T_r$ has $r$ points on each side and altogether has a total of $1+2 + \ldots + r = r(r+1)/2$ points. An easy exercise in plane geometry shows that the sum of three distances from each given point to the sides is a constant. This
suggests that we may label these $r(r+1)/2$ points in $T_n$ by $(x, y, z)$-coordinates
with the constraint that $x+y+z = r-1$. In particular, the three vertices of $T_n$ will be labeled by $(r-1, 0, 0)$, $(0, r-1, 0)$, and $(0,0, r-1)$, respectively; and each
other point in $T_r$ will be labeled accordingly by its normalized distances to the three sides of the triangle $T_r$.

\begin{lem} The equation $x+y+z = r-1$ has $r(r+1)/2$ integer solutions and further
all these integer solutions can be viewed as hexagonal lattice points in $T_r$.
\end{lem}

\begin{thm} \label{triangle}
If $n = r(r+1)/2$ for some integer $r$, then $f_3(n ) \le r$.
\end{thm}

\begin{proof}
Let $V=\{ (x, y, z) : x+y+z = r-1 \} $ be the set of $r(r+1)/2$
hexagonal lattice points in $T_r$.
Now we define three linear orders $\Pi =\{ \pi _1 , \pi _2, \pi _3\}$ on $V$ as follows:
$$\begin{array}{lllll}
(x_1, y_1, z_1) < (x_2, y_2, z_2) \mbox{ in } \pi _1 & \Longleftrightarrow  & x_1 < x_2 &\mbox{ or }& (x_1 =x_2 \mbox{ and } y_1 < y_2 ); \\
(x_1, y_1, z_1) < (x_2, y_2, z_2) \mbox{ in } \pi _2 & \Longleftrightarrow  & y_1 < y_2 &\mbox{ or }& (y_1 =y_2 \mbox{ and } z_1 < z_2 ); \\
(x_1, y_1, z_1) < (x_2, y_2, z_2) \mbox{ in } \pi _3 & \Longleftrightarrow  & z_1 < z_2 &\mbox{ or }& (z_1 =z_2 \mbox{ and } x_1 < x_2 ).
\end{array}$$
Since $\pi _1 , \pi _2, \pi _3$ are lexicographic orderings, they are all linear orders. Let $G_r$ be the $\{ \pi _1 , \pi _2, \pi _3\}$-majority tournament on $V$.
Then
$(x_1, y_1, z_1) < (x_2, y_2, z_2)$ in $G_r$ if and only if one of the following holds:
$$\begin{array}{llll}
1. & x_1 < x_2 & \mbox{and} & y_1 < y_2;\\
2. & y_1 < y_2 & \mbox{and} & z_1 < z_2;\\
3. & z_1 < z_2 & \mbox{and} & x_1 < x_2;\\
4. & x_1 = x_2 & \mbox{and} & y_1 < y_2;\\
5. & y_1 = y_2 & \mbox{and} & z_1 < z_2;\\
6. & z_1 = z_2 & \mbox{and} & x_1 < x_2.
\end{array}$$
Or, equivalently, the directions of the edges in
 $G_r$ (when viewing $V$ as hexagonal lattice points in $T_n$) can be displayed in a compass-like format.
(See Figure~\ref{Compass}.)
\begin{figure}[h]
\centering
\includegraphics[width = 2in, height = 2in]{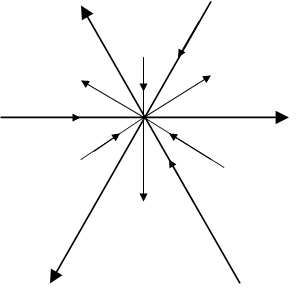}
\vspace{4mm}
\caption{Compass Showing Edge Directions in $G_r$}
\label{Compass}
\end{figure}

\vspace{4mm}

\begin{figure}[h]
\centering
\includegraphics[width = 2in, height = 1.7in]{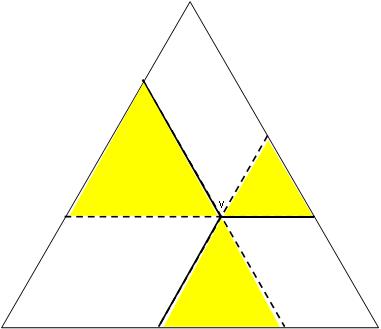}
\vspace{4mm}
\caption{Locations of the Remaining Vertices in the Sub-tournament}
\label{Induction}
\end{figure}

\vspace{2mm}

Claim 1: Any transitive sub-tournament of $G_r$ has at most $r$ vertices.

\vspace{4mm}

We will prove Claim 1 by induction on $r$. Obviously Claim 1 holds for $r=0$ and $1$.
Now suppose Claim 1 holds for all graphs $G_0, G_1, \ldots, G_{r-1}$.
Any transitive sub-tournament $D$ of $G_r$ has the ``biggest" vertex, say $(x_1, y_1, z_1)$,
which is adjacent to every other vertex in $D$. Through the point
$(x_1, y_1, z_1)$ in $T_r$ we draw three lines parallel to the sides, respectively.
Then $T_r$ is partitioned into three equilateral triangles and three parallelograms.
By the adjacency pattern shown in Figure~\ref{Compass}, the point $(x_1, y_1, z_1)$
is adjacent to every point in the three equilateral triangles (not including the dashed lines) and is adjacent from
every point in the three parallelograms (not including the solid lines inside $T_r$).
(See Figure \ref{Induction}.) Thus all
remaining vertices in $D$ must be contained in the three equilateral triangles shown with side lengths of $x_1-1, y_1-1,$ and $z_1-1$ (no vertices can appear on the dashed lines). However, these vertices in each individual equilateral triangle also form a transitive sub-tournament of the triangle;
that is, the sub-tournament of $G_{x_1}$, $G_{y_1}$, or $G_{z_1}$. Thus, by the induction hypothesis, the number of vertices in $D$ is at most $x_1+y_1+z_1 +1 = r$.
This proves Claim 1, from which Theorem~\ref{triangle} follows. 
\end{proof}

\begin{thm} \label{Main}
$f_3(n) < \sqrt {2n} +\frac 12. $
\end{thm}

\begin{proof}
Let $r$ be the unique integer satisfying $ r(r-1)/2 < n \le r(r+1)/2$.
Then $r(r-1)/2  \le n-1$ and thus $r < \sqrt {2n} +\frac 12.$
By Theorem \ref{triangle},
$$f_3(n) \le f(r(r+1)/2) \le r < \sqrt {2n} +\frac 12 . $$
\end{proof}

\vspace{.2cm}

{\bf Acknowledgement.} This is part of a research project done by three high school students (Alexandra Ilic, Lilly Shen, and Bobby Shen) in the summer of 2011
under the supervision of
Dr.~Jian Shen at Texas State University. Alexandra, Lilly, and Bobby
thank Texas State Honors Math Camp for providing this research opportunity.

\bibliographystyle{amsplain}

\end{document}